\documentclass[11pt]{article}
\usepackage{latexsym,amsfonts,amsthm,amsmath,amscd,amssymb,color,url}
\usepackage{tikz}
\usetikzlibrary{cd}

\newtheorem{lemma}{Lemma}[section]

\newtheorem{prop}[lemma]{Proposition}

\newcommand\matC{{\mathbb{C}}}

\newcommand\Sigmatil{{\widetilde\Sigma}}

\newcommand\htil{{\widetilde{h}}}

\renewcommand{\hbar}{{\overline{h}}}

\newfont{\Got}{eufm10 scaled 1200}

\newcommand{\compo}{\,{\scriptstyle\circ}\,}

\newcommand{\id}{\mathrm{id}}

\newcommand\calD{{\mathcal D}}

\newcommand\calR{{\mathcal R}}

\newcommand\calG{{\mathcal G}}

\begin{document}

\title{Counting surface branched covers}

\author{Carlo~\textsc{Petronio}\thanks{Partially supported by INdAM through GNSAGA, by
MIUR through PRIN ``Real and Complex Manifolds: Geometry, Topology and Harmonic Analysis''
and by UniPI through PRA 2018/22 ``Geometria e Topologia delle variet\`a''}\and Filippo~\textsc{Sarti}}

\maketitle

\begin{abstract}\noindent
To a branched cover $f$ between orientable surfaces one can associate
a certain \emph{branch datum} $\calD(f)$, that encodes the combinatorics
of the cover. This $\calD(f)$ satisfies a compatibility
condition called the Riemann-Hurwitz relation. The old
but still partly unsolved Hurwitz problem asks whether for
a given abstract compatible branch datum $\calD$ there exists
a branched cover $f$ such that $\calD(f)=\calD$. One can actually
refine this problem and ask \emph{how many} these $f$'s exist,
but one must of course decide what restrictions one puts on such $f$'s,
and choose an equivalence relation up to which one regards them. And it turns out that quite a few natural choices
are possible. In this short note we carefully analyze all these
choices and show that the number of actually distinct ones is only three.
To see that these three choices are indeed different we employ
Grothendieck's \emph{dessins d'enfant}.

\smallskip

\noindent MSC (2010): 57M12.
\end{abstract}

\noindent
Quite some energy has been devoted in ancient and more recent times
to the computation of the so-called \emph{Hurwitz numbers}, namely
the numbers of equivalence classes of branched covers matching
given branch data, see for instance~\cite{Medn1, Medn2, GKL, KM1, KM2, KML, MSS, x1, x3}.
However there are several natural \emph{a priori} different
ways to define the Hurwitz numbers exactly. In fact, one can decide
to insist or not that a branched cover should respect the orientation,
and to require or not that the branching points should be some fixed marked ones.
Analogously, one can decide to include or not the orientation and/or the marking in the
notion of equivalence. In this paper we carefully
list and analyze all these possibilities, and we show that actually only
three of them are really distinct. Our results and techniques have a
rather elementary nature, but to the best of our knowledge no explicit account
exists in the literature of all the possible ways under which the
problem of counting surface branched covers could be faced.
Two of the three different notions of equivalence
however appear in~\cite{LZ}, see the end of Section~\ref{leq3:sec}
for a discussion on this point.

\section{Branched covers, branch data,\\ and the Hurwitz problem}
A surface branched cover is a
map $$f:\Sigmatil\to\Sigma$$
where $\Sigmatil$ and $\Sigma$ are closed and connected surfaces and $f$ is locally modeled
on maps of the form
$$(\matC,0)\ni z\mapsto z^m\in(\matC,0).$$
If $m>1$ the point
$0$ in the target $\matC$ is called a \emph{branching point},
and $m$ is called the local degree at the point $0$ in the source $\matC$.
There are finitely many branching points, removing which, together
with their preimages, one gets a genuine cover of some degree $d$.
If there are $n$ branching points and we order them in some arbitrary fashion,
the local degrees at the preimages of the $j$-th point form a partition $\pi_j$ of $d$,
and we define
$$\calD(f)=\big(\Sigmatil,\Sigma,d,n,\pi_1,\ldots,\pi_n\big)$$
to be \emph{a branch datum} associated to $f$ (the datum is not unique
because an order of the branching points has been chosen).
If $\pi_j$ has length $\ell_j$, the following Riemann-Hurwitz relation holds:
$$\chi\big(\Sigmatil\big)-(\ell_1+\ldots+\ell_n)=d\big(\chi\big(\Sigma\big)-n\big).$$
Let us now call \emph{abstract branch datum} a 5-tuple
$$\calD=\big(\Sigmatil,\Sigma,d,n,\pi_1,\ldots,\pi_n\big),$$
a priori not coming from any $f$,
and let us say it is \emph{compatible} if it satisfies the Riemann-Hurwitz relation.
(For a non-orientable $\Sigmatil$ and/or $\Sigma$ this relation
should actually be complemented with certain other necessary conditions,
but we restrict to an orientable $\Sigma$ in this paper, so we do not
spell out these conditions here.)

\paragraph{The Hurwitz problem}
The very old \emph{Hurwitz problem} asks which compatible abstract branch data $\calD$ are
\emph{realizable}, namely, such that there exists $f:\Sigmatil\to\Sigma$ with
$\calD(f)=\calD$ for some ordering of the branching points of $f$,
and which are \emph{exceptional} (non-realizable).
Several partial solutions
to this problem have been obtained over the time, and we quickly
mention here the fundamental~\cite{EKS}, the survey~\cite{Bologna}, and
the more recent~\cite{Pako, PaPe, PaPebis, CoPeZa, SongXu}.
In particular, for an orientable $\Sigma$ the problem has been shown
to have a positive solution whenever $\Sigma$ has positive genus.
When $\Sigma$ is the sphere $S$, many realizability and exceptionality
results have been obtained (some of experimental nature), but the general
pattern of what data are realizable remains elusive. One guiding
conjecture in this context is that \emph{a compatible branch datum is always
realizable if its degree is a prime number}. It was actually shown in~\cite{EKS}
that proving this conjecture in the special case of $3$ branching
points would imply the general case. This is why many efforts have
been devoted in recent years to investigating the realizability
of compatible branch data with base surface $\Sigma$ the sphere $S$ and having $n=3$
branching points. See in particular~\cite{PaPe, PaPebis} for some evidence
supporting the conjecture.

\section{As many as 12 counting methods}
In this section and in the next one,
$$\calD(f)=\big(\Sigmatil,\Sigma,d,n,\pi_1,\ldots,\pi_n\big)$$
will be a fixed but arbitrary abstract compatible branch datum.
Note that $\Sigmatil$ and $\Sigma$ are supposed to be specific surfaces,
they are not viewed up to any type of equivalence. Moreover we assume they
have a fixed orientation, and we also fix $n$ points $p_1,\ldots,p_n$ in $\Sigma$.
In the sequel we will call \emph{positive} a map that respects the orientation.

\medskip

We will say that a branched cover $f:\Sigmatil\to\Sigma$ is:
\begin{itemize}
\item A \emph{realization} of $\calD$ if $\calD(f)=\calD$
for some order of the branching points;
\item A \emph{positive realization} of $\calD$ if it is a
realization and away from the branching points $f$ is a positive
local homeomorphism;
\item A \emph{marked realization} if the branching points of
$f$ are $p_1,\ldots,p_n$ and the partition of $d$ given by the
local degrees of $f$ at the preimages of $p_j$ is $\pi_j$;
\item A \emph{marked positive realization} if it is both a
marked realization and a positive one.
\end{itemize}

Note that a marked realization of $\calD$ is a realization. We denote by
$$R(\calD)\qquad
R_+(\calD)\qquad
R_*(\calD)\qquad
R_{*,+}(\calD)$$
respectively the set of all the realizations of $\calD$,
the positive realizations, the marked realizations, the marked positive realizations.
For all $m,s$ and $\mu,\sigma$ where $m,\mu$ are empty or $*$, except that
$\mu$ must be empty if $m$ is, and $s,\sigma$ are empty or $+$, we define a quotient
$\calR_{m,s}^{\mu,\sigma}(\calD)$ of $R_{m,s}(\calD)$, where two realizations
$f,f':\Sigmatil\to\Sigma$ in $R_{m,s}$ are identified if there exists a
commutative diagram
\begin{equation}\label{equiv:diag}
\begin{tikzcd}
 \Sigmatil  \arrow{r}{\htil} \arrow{d}[left]{f} & \Sigmatil  \arrow{d}{f'}\\
  \Sigma \arrow{r}{h} & \Sigma
\end{tikzcd}
\end{equation}

\noindent with $h,\htil$ homeomorphisms, and:
\begin{itemize}
\item $h$ is the identity if $\mu=*$;
\item $h$ and $\htil$ are positive if $\sigma=+$.
\end{itemize}
So we have the quotients
\begin{equation*}
\begin{array}{llll}
\calR(\calD)        &   \calR^+(\calD)          &       \calR_+(\calD)        &   \calR_+^+(\calD)       \\
\calR_*(\calD)      &   \calR_*^+(\calD)        &       \calR_*^*(\calD)      &   \calR_*^{*,+}(\calD)   \\
\calR_{*,+}(\calD)  &   \calR_{*,+}^+(\calD)    &       \calR_{*,+}^*(\calD)  &   \calR_{*,+}^{*,+}(\calD).
\end{array}
\end{equation*}
From the definitions it immediately follows that the following maps are defined:
\begin{itemize}
 \item the quotient maps $q_{m,s}^\mu:\calR_{m,s}^{\mu,+}(\calD)\twoheadrightarrow \calR_{m,s}^\mu(\calD);$
 \item the quotient maps  $c_s^\sigma:\calR_{*,s}^{*,\sigma}(\calD)\twoheadrightarrow \calR_{*,s}^\sigma(\calD);$
 \item the inclusions  $j_m^{\mu,\sigma}:\calR_{m,+}^{\mu,\sigma}(\calD)\hookrightarrow \calR_m^{\mu,\sigma}(\calD);$
 \item the forgetful maps $o_s^\sigma:\calR_{*,s}^\sigma(\calD)\rightarrow \calR_s^\sigma(\calD).$
\end{itemize}
Therefore we get the following commutative diagram

\begin{equation*}
\begin{tikzcd}
 \calR_*^{*,+} \arrow[twoheadrightarrow]{rr}{c^+} \arrow[twoheadrightarrow]{dd}{q_*^*}\arrow[hookleftarrow]{dr}{j_*^{*,+}} & &  \calR_*^+ \arrow[twoheadrightarrow]{dd}[near start]{q_*}\arrow[hookleftarrow]{rd}{j_*^+}\arrow[rightarrow]{rr}{o^+}& & \calR^+\arrow[twoheadrightarrow]{dd}[near start]{q}\arrow[hookleftarrow]{rd}{j^+} &\\
& \calR_{*,+}^{*,+}\arrow[twoheadrightarrow,crossing over]{rr}[near start]{c_+^+} \arrow[twoheadrightarrow]{dd}[near start]{\hspace{0,1cm}q_{*,+}^*} & & \calR_{*,+}^+\arrow[twoheadrightarrow]{dd}[near start]{\hspace{0,1cm}q_{*,+}}\arrow[rightarrow, crossing over]{rr}[near start]{o_+^+}& & \calR_+^+\arrow[twoheadrightarrow]{dd}{q_+} \\
  \calR_*^* \arrow[twoheadrightarrow]{rr}[near start]{c}\arrow[hookleftarrow]{rd}{j_*^*}& & \calR_*\arrow[hookleftarrow]{dr}{j_*}\arrow[rightarrow]{rr}[near start]{o}& & \calR\arrow[hookleftarrow]{dr}{j} &  \\
&  \calR_{*,+}^*\arrow[from=uu,crossing over] \arrow[twoheadrightarrow]{rr}{c_+}& & \calR_{*,+}\arrow[from=uu,crossing over]\arrow[rightarrow]{rr}{o_+}& & \calR_+
\end{tikzcd}
\end{equation*}
where to save space we have omitted any explicit reference to $\calD$.

\section{At most three methods are distinct}\label{leq3:sec}
In this section we show that several of the maps
in the above commutative diagram
are actually bijections
or otherwise easily understood.

\begin{prop}\label{qstplusst:prop}
$q_{*,+}^*$ is a bijection.
\end{prop}

\begin{proof}
If $f,f'\in R_{*,+}$ are equivalent in $\calR_{*,+}^*(\calD)$ then there
exists a commutative diagram as~(\ref{equiv:diag}) with $h$ the identity.  But
$f$ and $f'$ are positive, so $\htil$ also is, whence $f$ and $f'$ are
equivalent in $\calR_{*,+}^{*,+}(\calD)$.
\end{proof}

\begin{prop}\label{j:bijection}
$j,\ j_*$ and $j_*^*$ are bijections.
\end{prop}

\begin{proof}
We spell out the argument for $j$, the other two cases are identical.
Fix a negative involutive automorphism $\rho$ of $\Sigmatil$. Our aim
is to define an inverse $\varphi:\calR(\calD)\to\calR_+(\calD)$ of $j$.
To do so for $f\in R(\calD)$ we set
$$\varphi([f])=\begin{cases}
             [f] & \text{ if }f\in R_+(\calD) \\
             [f\compo\rho] & \text{ if }f\not\in R_+(\calD)
             \end{cases}$$
and we prove the following:

\medskip\noindent\textsc{Fact: $\varphi$ is well-defined}\quad
Note first that $f\compo\rho\in R_+(\calD)$ if $f\not\in R_+(\calD)$, so
it makes sense to take $[f\compo\rho]\in\calR_+(\calD)$. Assume now $f,f'\in R(\calD)$
are equivalent in $\calR(\calD)$, so there exists a commutative diagram as~(\ref{equiv:diag}).
We have four cases:
\begin{itemize}
\item If $f,f'\in R_+(\calD)$ the same diagram shows that $f$ and $f'$ are equivalent in $\calR_+(\calD)$;
\item If $f\in R_+(\calD)$ and $f'\not\in R_+(\calD)$ then $f$ and $f'\compo\rho$ are equivalent in $\calR_+(\calD)$
because we have the commutative diagram
\begin{center}
\begin{tikzcd}
 \Sigmatil  \arrow{r}{\rho\compo\htil} \arrow{d}{f} & \Sigmatil  \arrow{d}{f'\compo \rho}\\
  \Sigma \arrow{r}{h} & \Sigma
\end{tikzcd}
\end{center}
(recall that $\rho$ is involutive, so $\rho^2$ is the identity of $\Sigmatil$);
\item If $f\not\in R_+(\calD)$ and $f'\in R_+(\calD)$ then $f\compo\rho$ and $f'$ are equivalent in $\calR_+(\calD)$
because we have the commutative diagram
\begin{center}
\begin{tikzcd}
 \Sigmatil  \arrow{r}{\htil \compo \rho} \arrow{d}{f\compo \rho} & \Sigmatil  \arrow{d}{f'}\\
  \Sigma \arrow{r}{h} & \Sigma;
\end{tikzcd}
\end{center}
\item If $f,f'\not\in R_+(\calD)$ then $f\compo\rho$ and $f'\compo\rho$ are equivalent
in $\calR_+(\calD)$
because we have the commutative diagram
\begin{equation}\label{rhohtilrho:diag}
\begin{tikzcd}
\Sigmatil  \arrow{rr}{\rho\compo\htil \compo \rho} \arrow{dd}{f\compo \rho} & & \Sigmatil  \arrow{dd}{f'\compo \rho}\\
& & \\
  \Sigma \arrow{rr}{h} & & \Sigma.
\end{tikzcd}
\end{equation}

\end{itemize}

\medskip\noindent\textsc{Fact: $\varphi$ is a left inverse of $j$}\quad
This is immediate: for $f\in R_+(\calD)$ we have $j([f])=[f]\in \calR(\calD)$,
whence $(\varphi\compo j)([f])=[f]\in\calR_+(\calD)$.

\medskip\noindent\textsc{Fact: $\varphi$ is a right inverse of $j$}\quad
Given $f\in R(\calD)$ we have two cases; if $f\in R_+(\calD)$ then $\varphi([f])=[f]\in\calR_+(\calD)$,
so $(j\compo\varphi)([f])=[f]\in\calR(\calD)$; if $f\not\in R_+(\calD)$ then
$\varphi([f])=[f\compo\rho]\in\calR_+(\calD)$,
so $(j\compo\varphi)([f])=[f\compo\rho]\in\calR(\calD)$, but we actually have that $f$ and $f\compo\rho$
are equivalent in $\calR(\calD)$, because we have the commutative diagram
\begin{center}
 \begin{tikzcd}
 \Sigmatil  \arrow{r}{\rho} \arrow{d}{f\compo \rho} & \Sigmatil  \arrow{d}{f}\\
  \Sigma \arrow{r}{\id} & \Sigma.
\end{tikzcd}
\end{center}
The proof is complete.
\end{proof}

\begin{prop}
$o,\ o_+,\ o^+$ and $o_+^+$ are bijections.
\end{prop}

\begin{proof}
We begin with $o$ and we construct its inverse $\varphi$.
Given $f\in R(\calD)$, suppose that
$f$ has branching points $x_1,\ldots,x_n$ ordered so that
the local degrees of $f$ over $x_j$ form the partition $\pi_j$ of $d$.
We then choose an automorphism $g$ of $\Sigma$ such that $g(x_j)=p_j$ for all $j$
and set $\varphi([f]=[g\compo f]\in\calR_*(\calD)$. We have the following:

\medskip\noindent\textsc{Fact: $\varphi$ is well-defined}\quad
Note first that indeed $g\compo f\in R_*(\calD)$. Moreover
$[g\compo f]\in\calR_*(\calD)$ is independent of $g$, because if
$g'$ is another automorphism of $\Sigma$ such that $g'(x_j)=p_j$ for all $j$
we have the commutative diagram
\begin{center}
\begin{tikzcd}
\Sigmatil  \arrow{r}{\id} \arrow{d}[left]{g\compo f} & \Sigmatil  \arrow{d}{g'\compo f}\\
\Sigma \arrow{r}{g'\compo g^{-1}} & \Sigma.
\end{tikzcd}
\end{center}
Suppose now that $f,f'\in R(\calD)$ are equivalent in $\calR(\calD)$, so there exists a
commutative diagram as~(\ref{equiv:diag}); let $x_1,\ldots,x_n$ and $x'_1,\ldots,x'_n$
be the branching points of $f$ and $f'$ ordered so that
the local degrees of $f$ over $x_j$ and those of $f'$ over $x'_j$
form the partition $\pi_j$ of $d$. Take automorphisms
$g$ and $g'$ of $\Sigma$ such that $g(x_j)=g'(x'_j)=p_j$ for all $j$.
Then $g\compo f$ and $g'\compo f'$ are equivalent in $\calR_*(\calD)$ because
we have the commutative diagram
\begin{center}
 \begin{tikzcd}
 \Sigmatil  \arrow{rr}{\htil} \arrow{dd}[left]{g\compo f} & &\Sigmatil  \arrow{dd}{g'\compo f'}\\
 & & \\
  \Sigma \arrow{rr}{g'\compo h\compo g^{-1}} & &\Sigma.
\end{tikzcd}
\end{center}

\medskip\noindent\textsc{Fact: $\varphi$ is a left inverse of $o$}\quad
Given $f\in R_*(\calD)$ we know that the branching points of
$f$ are $p_1,\ldots,p_n$ with associated partitions $\pi_1,\ldots,\pi_n$,
so in the definition of $\varphi$ we can take $g=\id$ and it readily follows that
$$(\varphi\compo o)([f])=\varphi(o([f]))=\varphi([f])=[\id\compo f]=[f].$$

\medskip\noindent\textsc{Fact: $\varphi$ is a right inverse of $o$}\quad
Take $f\in R(\calD)$ and suppose that $\varphi([f])$ is defined as $[g\compo f]$.
Then $$(o\compo\varphi)[f]=o(\varphi([f]))=o([g\compo f])=[g\compo f]$$
but $g\compo f$ is equivalent to $f$ in $\calR(\calD)$ because
we have the commutative diagram
\begin{center}
 \begin{tikzcd}
 \Sigmatil  \arrow{r}{\id} \arrow{d}[left]{ f} & \Sigmatil  \arrow{d}{g\compo f}\\
  \Sigma \arrow{r}{g} &\Sigma.
\end{tikzcd}
\end{center}

\medskip

This concludes the argument for $o$. Repeating it for $o_+$ only requires
to remark that the automorphism $g$ of $\Sigma$ used to define $\varphi(f)$
can always be chosen to be positive. The extension to $o^+$ and $o^+_+$
is straight-forward.
\end{proof}

\begin{prop}
There exist natural $2:1$ projections
$$\calR^+(\calD)\to\calR_+^+(\calD)\qquad \calR_*^{*,+}(\calD)\to\calR_{*,+}^{*,+}(\calD)$$
having $j^+$ and $j_+^{*,+}$ as sections, so
$\calR^+(\calD)$ can be identified to two copies of $\calR_+^+(\calD)$, and
$\calR_*^{*,+}(\calD)$ can be identified to two copies of $\calR_{*,+}^{*,+}(\calD)$.
\end{prop}

\begin{proof}
As in the proof of Proposition~\ref{j:bijection} we fix a negative involutive automorphism $\rho$ of $\Sigmatil$.
We spell out the proof for $j^+$; the version for $j_+^{*,+}$ is identical.
For $f\in R(\calD)$ we set
$$\varphi([f])=\begin{cases}
             [f] & \text{ if }f\in R_+(\calD) \\
             [f\compo\rho] & \text{ if }f\not\in R_+(\calD)
             \end{cases}$$
and we have the following:

\medskip\noindent\textsc{Fact: $\varphi:\calR^+(\calD)\to\calR_+^+(\calD)$ is well-defined}\quad
For $f\not\in R_+(\calD)$ we have $f\compo\rho\in R_+(\calD)$, so it
makes sense to take $[f\compo\rho]\in\calR_+^+(\calD)$.
If $f,f'\in R(\calD)$ are equivalent in $\calR^+(\calD)$ then
there exists a commutative diagram as~(\ref{equiv:diag}) with $h$ and $\htil$ positive.
It follows that either $f,f'\in R_+(\calD)$ or $f,f'\not\in R_+(\calD)$.
In the former case the same diagram shows that $f$ and $f'$ are equivalent
in $\calR_+^+(\calD)$. In the latter case a commutative diagram as~(\ref{rhohtilrho:diag})
shows that $f\compo\rho$ and $f'\compo\rho$ are equivalent
in $\calR_+^+(\calD)$.

\medskip\noindent\textsc{Fact: $\varphi$ is $2:1$}\quad
For $g\in R_+(\calD)$ we have that $\varphi^{-1}([g])$ certainly contains
$[g]$ and $[g\compo\rho]$. Conversely, suppose that $\varphi([f])=[g]$.
We have two cases. Either $f\in R_+(\calD)$, whence $\varphi([f])=[f]$,
which implies that $[f]=[g]$ in $\calR_+^+(\calD)$ and hence in particular
in $\calR^+(\calD)$. Or $f\not\in R_+(\calD)$, so $\varphi([f])=[f\compo\rho]$,
whence $[f\compo\rho]=[g]$ in $\calR_+^+(\calD)$; this implies that
we have a commutative diagram
\begin{center}
\begin{tikzcd}
\Sigmatil \arrow{r}{\htil} \arrow{d}{f\compo \rho} & \Sigmatil \arrow{d}{g}\\
 \Sigma \arrow{r}{h} & \Sigma
 \end{tikzcd}
\end{center}
whence
\begin{center}
\begin{tikzcd}
\Sigmatil \arrow{rr}{\rho\compo\htil\compo\rho} \arrow{dd}{f} && \Sigmatil \arrow{dd}{g\compo\rho}\\
& &\\
 \Sigma \arrow{rr}{h} && \Sigma,
\end{tikzcd}
\end{center}
therefore $[f]=[g\compo\rho]$ in $\calR^+(\calD)$. Now we note that $g$ and $g\compo\rho$
cannot be equivalent in $\calR^+(\calD)$ and the claimed fact is proved.

\medskip\noindent\textsc{Fact: $j^+$ is a section of $\varphi$}\quad
This is immediate, because for $g\in R_+(\calD)$ we have $j^+([g])=[g]$ and
$\varphi([g])=[g]$.
\end{proof}

The results proved so far imply that of the $12$ potentially
distinct $\calR_{m,s}^{\mu,\sigma}(\calD)$ only three remain to understand, because
$$\begin{array}{lclclcl}
\calR(\calD)             & = & \calR_+(\calD)                       & = & \calR_{*,+}(\calD)      & = & \calR_*(\calD) \\
\calR_+^+(\calD)         & = & \calR_{*,+}^+(\calD)                 &   &                         &   &                \\
\calR_{*,+}^{*,+}(\calD) & = & \calR_{*,+}^*(\calD)                 & = & \calR_*^*(\calD)        &   &                \\
\calR_*^+(\calD)         & = & \calR^+(\calD)                       & = & 2\times\calR_+^+(\calD) &   &                \\
\calR_*^{*,+}(\calD)     & = & 2\times\calR_{*,+}^{*,+}(\calD).     &   &                         &   &
\end{array}$$
Note that each $\calR_{m,s}^{\mu,\sigma}(\calD)$ is a finite set (see also Section~\ref{dessin:sec})
and does not carry any
significant algebraic structure, so only its cardinality actually matters. In the
rest of the paper we will concentrate on
$$\calR_{*,+}^{*,+}(\calD)\qquad\calR_+^+(\calD)\qquad \calR(\calD)$$
and prove that they can indeed differ from each other. Note that by construction
$$\#(\calR_{*,+}^{*,+}(\calD))\geqslant\#(\calR_+^+(\calD))\geqslant\#(\calR(\calD))$$
and they can only vanish simultaneously.

\medskip

We remark here that
our $\calR_{*,+}^{*,+}(\calD)$ is called in~\cite{LZ} the set of \emph{rigid equivalence classes}
of branched covers matching $\calD$, while $\calR_+^+(\calD)$ is called the set of \emph{flexible
equivalence classes}. From this viewpoint, we could call $\calR(\calD)$ the set of
\emph{very flexible equivalence classes}.

\section{Dessins d'enfant}\label{dessin:sec}
Graphs on surfaces have been used for a long time to understand
surface branched covers, see~\cite{LZ}; their version for the case where
$\Sigma$ is the sphere $S$ and the number $n$ of branching point is $3$
has been popularized by Grothendieck under the name of \emph{dessins d'enfant},
and pushed to its most ultimate algebraic consequences, see~\cite{Groth} and~\cite{Cohen}.
In this paper we will adopt a purely topological viewpoint, neglecting the algebraic one
altogether.

\medskip

Let us call \emph{graph} a cellular $1$-complex $\Gamma$, consisting of
vertices and edges (we do not insist that $\Gamma$ should be simplicial,
so loops and multiple edges are allowed). We say that $\Gamma$ is
\emph{bipartite} if its vertices are coloured black and white, and every edge
joins black to white.  We then call \emph{dessin d'enfant} on $\Sigmatil$
a bipartite graph $\Gamma$ embedded in $\Sigmatil$ so that the complement of
$\Gamma$ consists of open discs, called \emph{regions}. The \emph{length} of a region
of $\Gamma$ is now  the number of black (or white) vertices it is incident to
(with multiplicity).

\medskip

In the rest of this section we fix a branch datum $\calD$ with base surface $\Sigma$
the sphere $S$ and $n=3$ branching points, so
$$\calD=\big(\Sigmatil,S,d,3,\pi_1,\pi_2,\pi_3\big).$$
The results we prove here could be extracted from the literature~\cite{LZ},
but we spell them out for the sake of completeness. We will denote by
$G(\calD)$ the set of dessins d'enfant $\Gamma$ on $\Sigmatil$ such that the valences
of the black vertices of $\Gamma$ give the partition $\pi_1$ of $d$, the valences
of the white vertices give $\pi_2$, while the lengths of the regions of $\Gamma$ give $\pi_3$.
We define the quotient $\calG_*(\calD)$ of $G(\calD)$ under the action of the positive automorphisms of $\Sigmatil$.

\begin{prop}\label{Rstarstar:prop}
There is a natural bijection between
$\calR_{*,+}^{*,+}(\calD)$ and
$\calG_*(\calD)$.
\end{prop}

\begin{proof}
Take $f\in R_{*,+}(\calD)$ and recall that the branching points of $f$ are the fixed points
$p_1,p_2,p_3$, with associated partitions $\pi_1,\pi_2,\pi_3$ of $d$, in this order. Now
choose in $S$ a simple arc $\alpha$ joining $p_1$ to $p_2$, paint $p_1$ black and $\pi_2$ white,
and define $\Gamma=f^{-1}(\alpha)$, with vertex colours lifted through $f$. Since $\alpha$ is
unique up to isotopy fixed on $p_1$ and $p_2$, the dessin d'enfant
$\Gamma$ is well-defined up to positive automorphisms of $\Sigmatil$.
Moreover, by the definition of the equivalence relation giving $\calR_{*,+}^{*,+}(\calD)$, we see that
$[\Gamma]\in\calG_*(\calD)$ only depends on $[f]\in\calR_{*,+}^{*,+}(\calD)$.

\medskip

Conversely, given $\Gamma\in G(\calD)$, we can choose in $S$ a simple arc $\alpha$ joining $p_1$ to $p_2$,
and map $\Gamma$ continuously to $\alpha$ sending the black vertices to $p_1$, the white vertices to $p_2$,
and each edge bijectively onto $\alpha$. We can now extend this map continuously
to each complementary region $\Omega$ of $\Gamma$,
in such a way that $\Omega$ is mapped onto $S$ and there is only one branching point
in the interior of $\Omega$, mapped to $p_3$ with local degree equal to the length of $\Omega$.
Patching these extensions together we get a branched cover $f$ in $R_{*,+}(\calD)$, and it is a routine
matter to check that $[f]\in\calR_{*,+}^{*,+}(\calD)$ only depends on
$[\Gamma]\in\calG_*(\calD)$.

\medskip

The two maps $[f]\mapsto[\Gamma]$ and $[\Gamma]\mapsto[f]$ are the inverse of each other, whence the conclusion.
\end{proof}

We next define a further quotient $\calG_+(\calD)$ where two elements of $\calG_*(\calD)$ are
identified if two of their representatives can be obtained from each other by a combination of the
following moves $\Gamma\mapsto\Gamma'$:
\begin{itemize}
\item $\Gamma'$ equals $\Gamma$ as a graph, but the black and white colours of the vertices are switched;
\item $\Gamma'$ has the same black vertices as $\Gamma$, one white vertex in the interior of each
complementary region $\Omega$ of $\Gamma$, and disjoint edges contained in $\Omega$ joining this white vertex to
the black vertices adjacent to $\Omega$.
\end{itemize}
Note that a single move $\Gamma\mapsto\Gamma'$ might lead from a dessin in $G(\calD)$ to one outside
$G(\calD)$, but then a further move $\Gamma\mapsto\Gamma'$ might lead back to $G(\calD)$.
Note also that if $\pi_1,\pi_2,\pi_3$ are distinct then $\calG_+(\calD)=\calG_*(\calD)$.

\begin{prop}
There is a natural bijection between
$\calR_+^+(\calD)$ and $\calG_+(\calD)$.
\end{prop}

\begin{proof}
Recall that we have a natural surjection $c_+^+:\calR_{*,+}^{*,+}(\calD)\to\calR_{*,+}^+(\calD)$
and that $\calR_+^+(\calD)$ is identified to $\calR_{*,+}^+(\calD)$ via $o_+^+$.
Take $f,f'\in R_{*,+}(\calD)$ and suppose that $c_+^+([f])=c_+^+([f'])$. Then there exists
a diagram as~(\ref{equiv:diag}) with positive $h,\htil$. Choose in $S$ arcs
$\alpha$ and $\alpha'$ joining $p_1$ to $p_2$, with a black end at $p_1$ and a white end at $p_2$,
so the equivalence classes in $\calG_*(\calD)$ of $\Gamma=f^{-1}(\alpha)$ and $\Gamma'=f'^{-1}(\alpha')$
correspond to the elements $[f]$ and $[f']$ of $\calR_{*,+}^{*,+}(\calD)$. Now $h^{-1}(\alpha')$
need not be isotopic to $\alpha$ with isotopy fixed at the ends, because it will have its black
end at some $p_{i_1}$ and its white end at some $p_{i_2}$ (but note that we must have $\pi_{i_1}=\pi_1$
and $\pi_{i_2}=\pi_2$). However $\alpha'$ becomes $\alpha$, up to isotopy
fixed at the ends, by a combination of the following
moves applied to an arc:
\begin{itemize}
\item Switch the colours of the ends of the arc;
\item Supposing the black end is at $p_1$ and the white end is at $p_2$,
replace the arc by one with black end at $p_1$ and white end at $p_3$.
\end{itemize}
This implies that $\htil^{-1}(\Gamma')$ becomes $\Gamma$ by a combination of the moves defining
the equivalence in $\calG_+(\calD)$, so $[\Gamma]=[\Gamma']$ in $\calG_+(\calD)$.
The same construction shows that if $[\Gamma]=[\Gamma']$ in $\calG_+(\calD)$ then
$c_+^+([f])=c_+^+([f'])$, whence the conclusion.
\end{proof}

As a final step, we take the quotient $\calG(\calD)$ of $\calG_+(\calD)$ by allowing
the action of the negative automorphisms of $\Sigmatil$ as well. Then we easily have:

\begin{prop}
There is a natural bijection between
$\calR(\calD)$ and $\calG(\calD)$.
\end{prop}

\section{Precisely three methods are distinct}
By the results of the previous section, all we must show to conclude that
$$\calR_{*,+}^{*,+}(\calD)\qquad\calR_+^+(\calD)\qquad \calR(\calD)$$
are distinct in some cases, is to show that each of the natural projections
$$\calG_*(\calD)\twoheadrightarrow\calG_+(\calD)\qquad
\calG_+(\calD)\twoheadrightarrow\calG(\calD)$$
can fail to be a bijection. Note first that given
$$\calD=\big(\Sigmatil,S,d,3,\pi_1,\pi_2,\pi_3\big)$$
there are finitely many abstract bipartite graphs with valences of
the black vertices the entries of $\pi_1$ and valences of the white vertices
the entries of $\pi_2$. Each of these graphs embeds in $\Sigmatil$ in only a
finite number of ways up to automorphisms of $\Sigmatil$, so
$\calG_*(\calD)$ is always a finite set. Even if we will not need this here,
we note that variants of the argument based on dessins d'enfant prove
that $\calR_{*,+}^{*,+}(\calD)$ is always a finite set, so
$\calR_+^+(\calD)$ and $\calR(\calD)$ also are.

\medskip

Let us now consider the branch datum
$$\calD=\big(S,S,7,3,[3,2,1,1],[3,2,1,1],[7]\big).$$
An easy enumeration argument proves that $\calG_*(\calD)$ has
$9$ elements, with representatives $\Gamma_1,\ldots,\Gamma_9$ shown
in Fig.~\ref{example:fig}.
 \begin{figure}[ht]
 \begin{center}
  \includegraphics[scale=1.2]{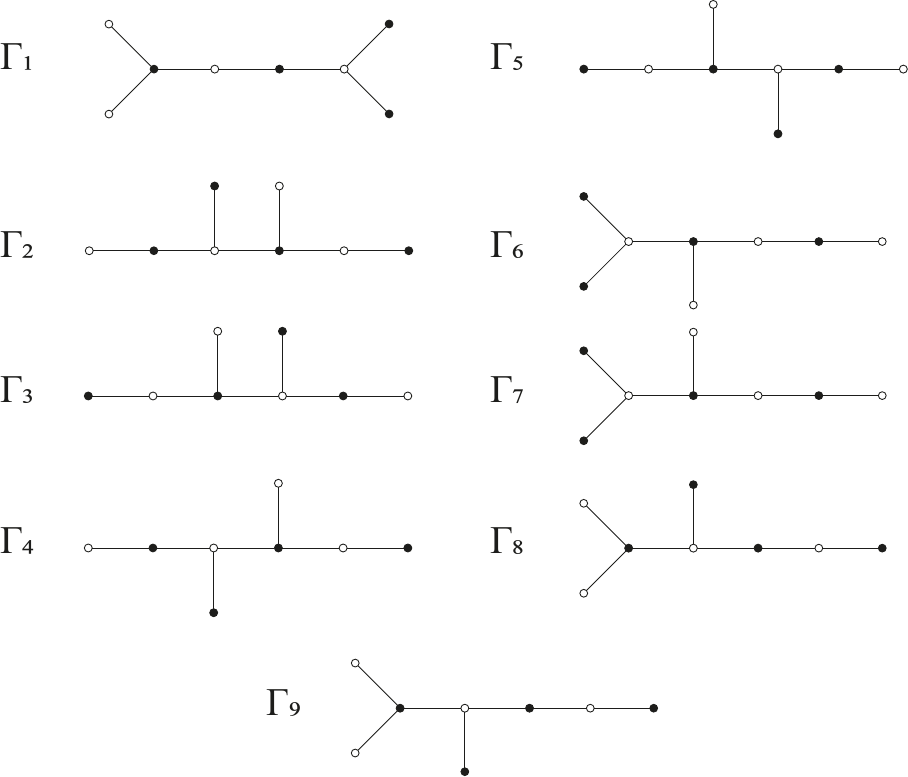}\caption{Representatives of $\calG_*(S,S,7,3,[3,2,1,1],[3,2,1,1],[7])$.}\label{example:fig}
  \end{center}
 \end{figure}

\medskip
Now one sees that the graphs $\Gamma_2$ and $\Gamma_3$ are
obtained from each other by a switch of colours, and the same happens for
$\Gamma_6$ and $\Gamma_9$, and for $\Gamma_7$ and $\Gamma_8$. Note
that only the switch of colours is relevant to
understand the projection $\calG_*(\calD)\to\calG_+(\calD)$, because
$\pi_1=\pi_2\neq \pi_3$,
so we conclude that $\calG_+(\calD)$ has $6$ elements.
We must now analyze what equivalence classes in $\calG_+(\calD)$ are
related by a reflection in $S$, and we see that it happens precisely for
those of $\Gamma_4$ and $\Gamma_5$, and for those of $\Gamma_6$ and $\Gamma_7$.
This implies that $\calG(\calD)$ has $4$ elements.  Therefore
\begin{eqnarray*}
&   & \#\big(\calR_{*,+}^{*,+}\big(S,S,7,3,[3,2,1,1],[3,2,1,1],[7]\big)\big)=9\\
& > & \#\big(\calR_+^+\big(S,S,7,3,[3,2,1,1],[3,2,1,1],[7]\big)\big)=6\\
& > & \#\big(\calR\big(S,S,7,3,[3,2,1,1],[3,2,1,1],[7]\big)\big)=4.
\end{eqnarray*}
Without providing the details we also note that in the relation
$$\#(\calR_{*,+}^{*,+}(\calD))\geqslant\#(\calR_+^+(\calD))\geqslant\#(\calR(\calD))$$
all the possibilities for the $\geqslant$ relations occur in reality; for instance
\begin{eqnarray*}
&   & \#\big(\calR_{*,+}^{*,+}\big(S,S,7,3,[7],[4,1,1,1],[3,2,1,1]\big)\big)\\
& = & \#\big(\calR_+^+\big(S,S,7,3,[7],[4,1,1,1],[3,2,1,1]\big)\big)=3\\
& > & \#\big(\calR\big(S,S,7,3,[7],[4,1,1,1],[3,2,1,1]\big)\big)=2,
\end{eqnarray*}
\begin{eqnarray*}
&   & \#\big(\calR_{*,+}^{*,+}\big(S,S,7,3,[3,3,1],[3,3,1],[4,2,1]\big)\big)=4\\
& > & \#\big(\calR_+^+\big(S,S,7,3,[3,3,1],[3,3,1],[4,2,1]\big)\big)\\
& = & \#\big(\calR\big(S,S,7,3,[3,3,1],[3,3,1],[4,2,1]\big)\big)=2,
\end{eqnarray*}
\begin{eqnarray*}
&   & \#\big(\calR_{*,+}^{*,+}\big(S,S,8,3,[4,2,2],[2,2,1,1,1,1],[8]\big)\big)\\
& = & \#\big(\calR_+^+\big(S,S,8,3,[4,2,2],[2,2,1,1,1,1],[8]\big)\big)\\
& = & \#\big(\calR\big(S,S,8,3,[4,2,2],[2,2,1,1,1,1],[8]\big)\big)=3.
\end{eqnarray*}

To conclude we note that the Hurwitz numbers computed in~\cite{Medn1, Medn2} are
$\#(\calR_{*,+}^{*,+}(\calD))$, while those computed in~\cite{x1, x3}
are $\#(\calR(\calD))$.

\end{document}